\documentclass[11pt,oneside]{article}
\usepackage{color}
\usepackage{psfrag}
\usepackage{epsfig}
\usepackage{graphicx}
\usepackage{amsmath}
\usepackage{amssymb}
\usepackage{amsthm}
\usepackage{cite}
\usepackage[latin1]{inputenc}

\newtheorem{theorem}{Theorem}[section]

\newtheorem{proposition}[theorem]{Proposition}
\newtheorem{lemma}[theorem]{Lemma}

\newtheorem{corollary}[theorem]{Corollary}

\begin{document}

\title{On the spectrum of sizes of semiovals contained in the Hermitian curve}
\date{}
\author{Daniele Bartoli\thanks{The author acknowledges the support of the European Community under a Marie-Curie Intra-European Fellowship (FACE project: number 626511).}, Gy\"orgy Kiss\thanks{The research was
supported by the Hungarian National
Foundation for Scientific Research, Grant NN 114614.}, \\
Stefano Marcugini\thanks{The research was supported by the Italian MIUR
(progetto 40\% ``Strutture Geometriche, Combinatoria e loro Applicazioni''),
and by GNSAGA (INDAM).},
 and Fernanda Pambianco$^\ddag$}

\maketitle

\begin{abstract}
Some constructions and bounds on the sizes of semiovals contained in  the Hermitian curve are given. A construction of an infinite family  of $2$-blocking sets of the Hermitian curve is also presented.
\end{abstract}

\section{Introduction}
Let $\Pi _q$ be a finite projective plane of order $q$
and let $\mathrm{PG}(2,q)$ denote the Desarguesian projective plane over the 
finite field of $q$ elements, $\mathbb{F}_q$. A {\it semioval}
$\mathcal{S}$ in $\Pi _q$ is a non-empty pointset  with the property
that for every point $P \in \mathcal{S}$ there exists a unique line $t_P$ such
that $\mathcal{S}\cap t_P= \{P\}$. This line is called the tangent line to $\mathcal{S}$ 
at $P$. A {\it blocking semioval} 
is a semioval $\mathcal{S}$ such that each line of $\mathrm{PG}(2,q)$
contains at least one point of $\mathcal{S}$ and at least one point outside $\mathcal{S}$. A blocking semioval existing in every projective plane of order $q>2$ is the vertexless triangle, the set of points formed by the union of three
non-concurrent lines with the points of intersections removed.

The classical examples of semiovals arise from polarities (ovals and unitals),
and from the theory of blocking sets. The semiovals
are interesting objects in their own right, but the study of semiovals is also
motivated by their applications to cryptography. Batten  constructed in
\cite{Batten2000} an effective message sending scenario which uses determining
sets. She showed that blocking semiovals are a particular type of determining
sets in projective planes.

It is known that if $\mathcal{S}$ is a semioval in $\Pi _q$ then 
$q+1\leq |\mathcal{S} | \leq
q\sqrt{q}+1$ and both bounds are sharp \cite{Hubaut1970,Thas1974}; the
extremes occur when $\mathcal{S}$ is an oval or a 
unital. In particular in $\mathrm{PG}(2,q)$ a conic has $q+1$ points
and if $q$ is a square then a Hermitian curve has $q\sqrt{q}+1$ points. 
A survey on results about semiovals
can be found in \cite{Kiss2007}.

In the last years the interest and research on the fundamental problem of
determining the spectrum of the values for which there exists a given
subconfiguration of points in $\mathrm{PG}(n,q)$ have increased considerably
(see for example 
\cite{BDFMP2012,Bartoli2014,BMP2013b,BE1999,CS2012,DFMP2009,HS2001,KMP2010,MMP2005,MMP2007,PS2008}).
For $q\leq 9$, $q$ odd, the spectrum of sizes of semiovals was determined 
by Lisonek \cite{Lisonek1994} by exhaustive computer search.
Kiss, Marcugini, and Pambianco \cite{KMP2010} extended Lisonek's results  
to the cases $q=11$ and $13.$ 

There are many known constructions and theoretical results about
semiovals, in particular those that either contain large collinear subsets 
in which case their size is close to the lower bound, or their size is close
to the upper bound. In the latter case 
Kiss, Marcugini and Pambianco \cite{KMP2010} constructed semiovals by 
careful deletion of points from a unital. If $q$ is an odd square then
they gave explicit examples of
semiovals of size $k$ for all $k$ satisfying the inequalities
$q(\sqrt{q}+1)/2\leq k\leq q\sqrt{q}+1$ and they proved existence 
if $q(\lceil4\log q\rceil +1)\leq k\leq q\sqrt{q}+1$ holds. The unital 
they started from was originally constructed by Sz\H onyi \cite{SZ1992}. 
It is a pencil of superosculating conics which is also a minimal blocking set. Later on, Dover and Mellinger gave the complete characterization of
semiovals from unions of conics \cite{DM2011}.

Our goal in this paper is to give similar constructions and estimates on 
the sizes of semiovals coming from the classical unital, the Hermitian curve. 
The main result is an explicit construction of semiovals of size $k$ for all $k$ 
satisfying the inequalities 
$2q\sqrt[4]{q}+4q-2\sqrt[4]{q^3}-3\sqrt{q}+1\leq k\leq q\sqrt{q}+1$ 
if $q=s^4$ and odd (see Corollary \ref{best}).   
We also present explicit examples for other values of $q$ and prove the existence existence of semiovals of size k for $q$ odd if 
$(q-\sqrt{q}+1)\left\lceil \frac{4(\sqrt{q}+1)}{\sqrt{q}-1}\log q\right\rceil \leq k\leq q\sqrt{q}+1$. 
If $q$ is large enough, then this gives a slight improvement on the previously known  
bound for almost all $q.$ Our main tools  are the application of proper blocking sets of 
the Hermitian curve constructed by Blokhuis at al. \cite{BJKRSSV2015} and the decomposition of the  
Hermitian curve into 
a union of $(q-\sqrt{q}+1)$-arcs, originally given by Seib \cite{S1970}, 
see in English in \cite{FHT1986}.

Finally, in the last section, we present a construction of an infinite family 
of $2$-blocking sets of the Hermitian curve.

\section{Explicit constructions of semiovals}

For the sake of convenience from now on we work on planes of order $q^2$. 
In this section we construct various examples of semiovals in $\mathrm{PG}(2,q^2)$
arising from the points of the Hermitian curve $\mathcal{H}_q$. 
This curve has $q^3+1$ points, there is a unique tangent line 
to $\mathcal{H}_q$ at each of its points and each of the other
$q^4-q^3+q^2$ lines of $\mathrm{PG}(2,q^2)$ is a $(q+1)$-secant of
$\mathcal{H}_q.$ 
A pointset $\mathcal{D}\subset \mathcal{H}_q$ is called a 
\emph{$2$-blocking set} if each $(q+1)$-secant contains at least 2 points 
of $\mathcal{D}.$ For  a detailed description of $\mathcal{H}_q$ we 
refer to \cite{HirsBook}.

For $q=2,3$ we could perform exhaustive computer search and the situation is the following. In $\mathrm{PG}(2,4)$, semiovals contained in the Hermitian curve exist only of sizes $6,8,9$: this means that Theorem \ref{th:primo} gives the complete spectrum of semiovals contained in the Hermitian curve for $q=2$. The spectrum of the sizes of semiovals contained in the Hermitian curve of $\mathrm{PG}(2,9)$ and the number of non-equivalent examples (up to collineations) are presented in Table \ref{PG2_9}.
\begin{center}
\begin{table}
\caption{Semiovals contained in the Hermitian curve of $\mathrm{PG}(2,9)$}\label{PG2_9}
\vspace*{0.1 cm}
\begin{center}
\begin{tabular}{|c||cccccccccccccc|}
\hline
Sizes&$12$&$15$&$16$&$18$&$19$&$20$&$21$&$22$&$23$&$24$&$25$&$26$&$27$&$28$\\
\hline
Non-equivalent examples &$1$&$1$&$2$&$5$&$4$&$9$&$10$&$5$&$8$&$6$&$2$&$1$&$1$&$1$\\
\hline
\end{tabular}
\end{center}
\end{table}
\end{center}

For the constructions we need the following elementary observations.

\begin{proposition}
\label{torol} 
Let $\mathcal{S}$ be a semioval. 
Suppose that the pointset $\mathcal{T}\subset \mathcal{S} $
has the property that if $\ell $ is a secant line to 
$\mathcal{S} $ then the inequality 
$|\mathcal{S}\cap \ell | \geq |\mathcal{T}\cap \ell | +2$ holds.
Then $\mathcal{S} \setminus \mathcal{T}$ is a semioval and 
$\mathcal{S}$ contains semiovals of size $k$ for all $k$ 
satisfying the inequalities 
$|\mathcal{S}\setminus \mathcal{T}|\leq k\leq \mathcal{S}$. 
\end{proposition}

\proof 
Let $R$ be a point of $\mathcal{S} \setminus \mathcal{T}.$
Then the tangent to $\mathcal{S} $ at $R$ is obviously a tangent 
to $\mathcal{S} \setminus 
\mathcal{T}$ at $R.$ We have to prove  that no new
tangents appear after the deletion of points of $\mathcal{T}.$ But if a line 
$\ell $ meets $\mathcal{S} $ in more than one point, then 
$$|(\mathcal{S} \setminus \mathcal{T})\cap \ell |=|\mathcal{S}\cap \ell | -|\mathcal{T}\cap \ell |\geq 2.$$
Thus no former secant line becomes tangent line to 
$\mathcal{S} \setminus \mathcal{T}$, so it is a semioval. 

Let  
$|\mathcal{S}\setminus \mathcal{T}|=k_0.$
If $\mathcal{U}$ is any subset of $k-k_0$ points of 
$\mathcal{T}$ then $((\mathcal{S}\setminus \mathcal{T})\cup \mathcal{U})
\subset \mathcal{S}$
is a semioval of size $k.$
\endproof

\begin{corollary}
\label{2-blocking} 
Let $\mathcal{B}$ be a 2-blocking set of $\mathcal{H}_q.$ 
Then in $\mathrm{PG}(2,q^2)$ there exist semiovals of size $k$ for all $k$ 
satisfying the inequalities 
$|\mathcal{B}|\leq k\leq q^3+1.$ 
\end{corollary}
 
\proof 
The set $\mathcal{T}=\mathcal{H}_q \setminus \mathcal{B}$ satisfies 
the condition of Proposition \ref{torol}.
\endproof

Our first, obvious construction does not depend on the parity of $q.$
\begin{theorem}\label{th:primo}
Let $q\geq2$. In $\mathrm{PG}(2,q^2)$ 
there exists a semioval $\mathcal{S}\subset \mathcal{H}_q$ of size $k$ 
for all $k\in \{ q^3-q^2+q\} \cup [q^3-q^2+q+2,q^3+1]$.
\end{theorem}

\proof
Let $P$ be a point in $\mathcal{H}_q$ and $\ell_1,\ell_2,\dots ,\ell_{q-1}$
be $(q+1)$-secants through $P.$  
Let $\mathcal{T}=\bigcup_{i=1}^{q-1}\ell_i$. Then $\mathcal{H}_q\setminus \mathcal{T}$ is a 
semioval of size $q^3-q^2+q$ because if $\ell $ is a $(q+1)$-secant of
$\mathcal{H}_q$ then we either deleted all of its points, or at
most $q-1$ of its points. Hence no former secant line becomes a tangent line to
$\mathcal{H}_q\setminus \mathcal{T},$ and there is exactly one tangent line at each point of $\mathcal{H}_q\setminus \mathcal{T},$ then $\mathcal{H}_q\setminus \mathcal{T}$ is a semioval of size $q^3-q^2+q$. Also, we can add $\overline{k}\in [2,q^2-q+1]$ points from $\mathcal{T}$ in a way that no other tangent lines are created. In fact, it is sufficient to control that in each line $\ell_1,\ell_2,\dots ,\ell_{q-1}$ the number of added points is different from $1$. This is always possible since $\overline{k}\neq 1$.
\endproof

The next two algebraic constructions work for all $q,$ but the lower bound 
of the size depend on the parity of $q.$
We use the following description of 
$\mathcal{H}_q\subset \mathrm{PG}(2,q^2)$. 
The curve $\mathcal{H}_q$ is defined by the equation
\begin{equation}\label{hermitian}
X_2X_0^q+X_2^q X_0+X_1^{q+1}=0.
\end{equation}
Let $c\in \mathbb{F}_{q^2}$ be a fixed root of the equation $c^q+c+1=0.$ 
Consider the set
\begin{equation}\label{M}
M=\{m \in \mathbb{F}_{q^2}\; |\; m^q+m=0\},
\end{equation}
then the points of $\mathcal{H}_q$ are
\begin{equation}\label{U}
\{ (1:u:cu^{q+1}+m) \mid  u \in \mathbb{F}_{q^2}, m \in M\}\cup\{(0:0:1)\}.
\end{equation}

If $q$ is an odd prime power then
let $h$ be a fixed non-square in $\mathbb{F}_q$ and consider 
$\mathbb{F}_{q^2}=\mathbb{F}_q[i]$ where
$i^2 = h.$ Then $i^q +i = 0$, $i^2 = i^{2q}$ and $i^{q+1} = -h.$ 

If $q$ is even then 
let $h\in \mathbb{F}_q$ be an element with 
$\mathrm{Tr}_{\mathbb{F}_{q}/\mathbb{F}_{2}}(h)=1$. Let $i^2 + i + h = 0$,
and consider $\mathbb{F}_{q^2}=\mathbb{F}_q[i].$
Then $i^q +i = 1$, $i^2 = i+h$ and $i^{q+1} = h.$

For all $q$ 
we represent the elements $x$ of $\mathbb{F}_{q^2}$ 
as $x = x_1 + ix_2$ where $x_1, x_2 \in  \mathbb{F}_q$.

\begin{theorem}\label{th:secondo}
If $q\geq3$ 
then in $\mathrm{PG}(2,q^2)$ 
there exists a semioval $\mathcal{S}\subset \mathcal{H}_q$ of size $k$ 
for all $k$ satisfying the inequalities 
$$q^3+1\geq k\geq
\begin{cases}
q^3-2q^2+4q+1 & \text{if  } q \text{ is even},\\
q^3-2q^2+5q-2 & \text{if } q \text{ is odd}.
\end{cases}
$$
\end{theorem}

\proof
Take the subset
\begin{equation}\label{Ubarra}
\mathcal{T} =  \{(1:u:cu^{q+1}+m)\mid u \in \mathbb{F}_{q}, m \in M, m \neq 0\}
\end{equation}
of points of $\mathcal{H}_q.$ Note that $|\mathcal{T}|=q(q-1)$. 
The condition $u \in \mathbb{F}_{q}$ implies that $u^{q+1}=u^2,$
thus the points of $\mathcal{T}$ can be written as $(1:u:cu^2+m),$ too.   
We claim that the set 
$\mathcal{U}=\mathcal{H}_q \setminus \mathcal{T}$ is a semioval. 
Because of Proposition \ref{torol} it is enough to prove that any line of $\mathrm{PG}(2,q^2)$ contains 
at most $q-1$ points of $\mathcal{T}$.

First consider the lines throught the point $P=(0:0:1).$ The 
line $X_0=0$ is the tangent to $\mathcal{H}_q$ at $P,$ while   
a line $\ell _{\alpha }$ having equation $X_1= \alpha X_0$ 
meets $\mathcal{T}$ in points whose second coordinate is $\alpha .$
Thus $\ell _{\alpha }$ contains $q-1$ 
points of $\mathcal{T}$ if $\alpha \in \mathbb{F}_{q}$ and no 
points of $\mathcal{T}$ if $\alpha \notin \mathbb{F}_{q}.$

Now consider the other lines of the plane.
If a line $\ell $ does not contain $P$ then its equation can be written as 
$\alpha X_0 +\beta X_1 +X_2 =0$.
If the point $(1:u:cu^2+m)$ is on $\ell $ then
$\alpha +\beta u+c u^2+m =0,$
hence we get 
$$m=- c u^2- \beta u -\alpha 
\quad \mathrm{and} \quad 
m^q=- c^q u^2- \beta^q u -\alpha^q .$$ 
But $m$ satisfies the condition $m^q+m=0$ hence
\begin{equation}
\label{masodfok}
-(c^q +c)u^2 - \left( \beta^q +\beta \right) u -\alpha^q  -\alpha =0.
\end{equation}
The coefficient of $u^2$ is $1$ since $c^q+c+1=0$.
So Equation (\ref{masodfok}) is a quadratic equation on $u,$ 
it has at most two roots. 
Hence $\ell $
contains at most $2$ points of $\mathcal{T}$ and $2\leq q-1$ because $q>2.$

In the second step
we get smaller semiovals by careful deletion of points 
of $\mathcal{U}$.

Let $M=\{0,m_1,\ldots,m_{q-1}\}$ and for $j=1,2,\dots , q-1$ let $\ell_{j}$
be the line with equation $X_2=(c+m_j)X_0.$ 
Then $\ell_{j}$ passes on the point $(0:1:0).$ 
The point $\ell_{\alpha }\cap \ell_j $ has coordinates 
$(1:\alpha: c+m_j),$ thus it belongs to $\mathcal{T}$
if and only if $\alpha \in \mathbb{F}_{q}$ and $\alpha ^{q+1}=1$ hold 
simultaneously.
It happens if and only if $\alpha ^{2}=1,$ hence $\ell_{j}$ 
contains $q-1$ or $q$ points of $\mathcal{H}_q \setminus \mathcal{T}$ 
if $q$ is odd or even, respectively. 
If $q=3$ then $\mathcal{U}$ is a semioval of size $q^3-q^2+q+1=22$. Let $q>3$ and consider  $\mathcal{V}=\bigcup_{j=1}^{q-3}\ell_j$. 
Then $\mathcal{S}_0=\mathcal{U} \setminus \mathcal{V}$ is a semioval: 
in fact each $(q+1)$-secant of $\mathcal{H}_q$ contains at most two 
points of $\mathcal{T}$ and 
$q-3$ points of $\mathcal{V}$ hence it is not a tangent to $\mathcal{S}_0$. 

The size of $\mathcal{S}_0$ is $q^3-q^2+q+1-(q-1)(q-3)=q^3-2q^2+5q-2$ if $q$ 
is odd and $q^3-q^2+q+1-q(q-3)=q^3-2q^2+4q+1$ if $q$ is even.
Note that, contrary to Theorem \ref{th:primo}, we can add also one point to $\mathcal{S}_0$: in fact it is sufficient to add a point in $\mathcal{T}$ and the new set is still a semioval.
\endproof

\begin{theorem}\label{th:terzo}
If $q>2$ 
then in $\mathrm{PG}(2,q^2)$ 
there exists a semioval $\mathcal{S}\subset \mathcal{H}_q$ of size $k$ 
for all $k$ satisfying the inequalities 
$$q^3+1\geq k\geq
\begin{cases}
\frac{q^3+2q^2-q+2}{2} & \text{if  } q \text{ is odd},\\
\frac{q^3+3q^2-2q+2}{2} & \text{if } q \text{ is even}.
\end{cases}
$$
\end{theorem}

\proof
Consider $\mathbb{F}_{q^2}=\mathbb{F}_q[i].$ 
Let $v\in \mathbb{F}_{q}$ be a fixed element. 
Take the subset
\begin{equation}\label{Ubarra_v}
\mathcal{T}_v =  \{(1:u+iv:c(u+iv)^{q+1}+m)\mid u \in \mathbb{F}_{q}, m \in M, m \neq 0\}
\end{equation}
of points of $\mathcal{H}_q.$
Note that $|\mathcal{T}_v|=q(q-1)$.

If a line $\ell $ does not contain the point $P=(0:0:1)$ 
then its equation can be written as 
$\alpha X_0 +\beta X_1 +X_2 =0$.
We claim that $\ell $ contains at most two points of $\mathcal{T}_v$. 

First consider the case $q$ odd. The condition 
$u,v \in \mathbb{F}_{q}$ implies that $(u+iv)^{q+1}=u^2-h^2v^2,$ thus the 
points of $\mathcal{T}_v$ can be written as $(1:u+iv:c(u^2-h^2v^2)+m),$ too.
If the point $(1:u:c(u^2-h^2v^2)+m)$ is on $\ell $ then
$\alpha +\beta (u+iv)+c (u^2-h^2v^2)+m =0.$
Rearranging this we get 
$$m=- c u^2- \beta u -\alpha 
- \beta iv + ch^2v^2
\quad \mathrm{and} \quad 
m^q=- c^q u^2- \beta^q u -\alpha^q 
+ \beta^q iv + c^qh^2v^2.$$ 
But $m$ satisfies the condition $m^q+m=0$ hence
\begin{equation}
\label{masodfok_v}
-(c^q +c)u^2- \left( \beta^q +\beta \right) u 
-\alpha^q  -\alpha 
- (\beta + \beta^q )iv +(c^q+c)h^2v^2 =0.
\end{equation}

If $q$ is even then we get a similar equation. The condition 
$u,v \in \mathbb{F}_{q}$ now implies that $(u+iv)^{q+1}=u^2+hv^2,$ thus the 
points of $\mathcal{T}_v$ can be written as $(1:u+iv:c(u^2+hv^2)+m).$ 
If the point $(1:u:c(u^2+hv^2)+m)$ is on $\ell $ then
$$m=c u^2 + \beta u + \alpha + \beta iv + chv^2
\quad \mathrm{and} \quad 
m^q= c^q u^2 +\beta^q u +\alpha^q 
+ \beta^q iv + +\beta ^qv+ c^qhv^2.$$ 
But $m$ satisfies the condition $m^q+m=0$ hence
\begin{equation}
\label{masodfok_veven}
(c^q +c)u^2 +\left( \beta^q +\beta \right) u 
+\alpha^q  +\alpha 
+ (\beta + \beta^q )iv +\beta ^qv+ (c^q+c)hv^2 =0.
\end{equation}

The coefficient of $u^2$ in Equations (\ref{masodfok_v}) 
and (\ref{masodfok_veven}) is $1$ since $c^q+c+1=0$.
So these are quadratic equations on $u,$ each of them has at most two roots. 
Hence $\ell $ contains at most $2$ points of $\mathcal{T}_v.$ 

Let $v_1,v_2,\dots ,v_{\lfloor (q-1)/2\rfloor }$ be distinct elements of 
$\mathbb{F}_{q}$ and let
$\mathcal{T}=\bigcup_{i=j}^{\lfloor (q-1)/2\rfloor }\mathcal{T}_{v_j}$. 
We show that $\mathcal{S}_0=\mathcal{H}_q\setminus \mathcal{T}$ is a 
semioval.

Because of Proposition \ref{torol} it is enough to prove that any 
$(q+1)$-secant of $\mathcal{H}_q$ contains at most $q-1$ points of 
$\mathcal{T}$. This is obvious if a line does not contain the point 
$P$ because in this case it contains at most two points from each 
set $\mathcal{T}_{v_j}$.
Consider the lines throught the point $P=(0:0:1).$ 
The line $X_0=0$ is the tangent
to $\mathcal{H}_q$ at $P,$ while   
a line $\ell _{\alpha }$ having equation $X_1= \alpha X_0$ 
meets $\mathcal{T}_{v_j}$ in points whose second coordinate is $\alpha .$
Thus $\ell _{\alpha }$ contains $q-1$ 
points of $\mathcal{T}_{v_j}$ if $(\alpha -iv_j)\in \mathbb{F}_{q}$ and no 
points of $\mathcal{T}_{v_j}$ if $(\alpha -iv_j)\notin \mathbb{F}_{q}.$

The size of $\mathcal{S}_0$ is $q^3+1-q(q-1)^2/2$ if $q$ 
is odd and $q^3+1-q(q-1)(q-2)/2$ if $q$ is even.
Thus the theorem follows from Proposition \ref{torol}.

\endproof

If $q-1$ has suitable divisors then we can construct smaller semiovals 
than in Theorem 
\ref{th:secondo}.
We distinguish the cases $q$ even and $q$ odd. 

First let $q$ be an odd prime power.
The following lemma due to Blokhuis et al. \cite{BJKRSSV2015}
gives information on the irreducibility of a particular plane curve. 
\begin{lemma}[\!\!\cite{BJKRSSV2015}, Lemma 4.5]\label{Lemma_odd}
Let $q$ be odd. If $n_1^2\neq 2d_1 + h n_2^2$
then the algebraic curve in $\mathrm{PG}(2,q)$ defined as 
\begin{equation}\label{eq:Curve}
\mathcal{X}_{\mathrm{odd}} : 2n_1X_0^rX_2^r+2hn_2X_1X_2^{2r-1}+2d_1X_2^{2r}+X_0^{2r} - hX_1^2X_2^{2r-2} = 0
\end{equation}
is absolutely irreducible and it has genus $g=r-1$.
\end{lemma}

If $q$ is even then a similar lemma holds.
\begin{lemma}[\!\!\cite{BJKRSSV2015}, page 14]\label{Lemma_even}
Let $q$ be even. If $n_1^2\neq n_1n_2 + h n_2^2 +d_2 \neq 0$
then the algebraic curve in $\mathrm{PG}(2,q)$ defined as 
\begin{equation}\label{eq:Curve_qeven}
\mathcal{X}_{\mathrm{even}} : (n_1 +n_2)X_1X_2^{2r-1} +n_2X_0^rX_2^r +d_2X_2^{2r} +X_0^{2r} +X_0^rX_1X_2^{r-1} +hX_1^2X_2^{2r-2} = 0
\end{equation}
is absolutely irreducible and it has genus $g\leq r-1$.
\end{lemma}

Using these lemmas, Blokhuis et al. \cite{BJKRSSV2015} constructed a 
blocking set of $\mathcal{H}_q.$ With a slight modification of their proof 
we can prove the existence of a family of 
semiovals contained in the Hermitian curve.

\begin{theorem}\label{Th:Main}
Let $q$ be a prime power and $r$ be a divisor of 
$q-1$ for which $r <\frac{\sqrt{q}}{2}$ holds. Then in
$\mathrm{PG}(2,q^2)$ there exists a semioval 
$\mathcal{S}\subset \mathcal{H}_q$ of size $k$ 
for all $k$ satisfying the inequalities  
\begin{equation}\label{eq:k}
\left(q+\frac{(q-1)q}{r}\right)(q-1) +q^2+1\leq k \leq q^3+1.
\end{equation}
\end{theorem}
\proof
The equation of $\mathcal{H}_q$ 
is the same
$$X_2X_0^q+X_2^q X_0+X_1^{q+1}=0$$
as in the previous proof, but we use another 
description of its points.

The point $X_{\infty}=(1:0:0)$ is on $\mathcal{H}_q$ and the line $\ell_{\infty} \, :\, X_2=0$ is the tangent to
$\mathcal{H}_q$ at $X_{\infty}.$  Choose $\ell_{\infty}$ as line at infinity and consider 
$\mathrm{PG}(2,q^2)$ as the union of $\mathrm{AG}(2,q^2)$ and $\ell_{\infty}.$
Let $\mathcal{U}$ be the affine part of $\mathcal{H}_q.$ Then the equation of $\mathcal{U}$ is  
\begin{equation}\label{eq:HC}
X^q + X + Y^{q+1} = 0.
\end{equation}
The tangent to $\mathcal{H}_q$ at the affine point $(a,b)$ has equation $X=-b^qY-a^q.$ Thus 
a non-horizontal affine line $X=nY+d$ is a tangent to $\mathcal{H}_q$ if and only if $n^{q+1}\neq d^q + d.$

Take the following pointset  
\begin{equation}\label{eq:B}
\mathcal{B} = \{(x, y) \in  \mathcal{U} \mid y = u^r + iv, u, v \in \mathbb{F}_q \} \cup  \{(1 : 0 : 0)\}\subset \mathcal{H}_q.
\end{equation}
For all $u, v \in \mathbb{F}_q$ the horizontal affine line with equation $Y = u^r + iv$ 
contains $q$ affine points of $\mathcal{B}$, 
the other horizontal lines $Y= y_0\neq y = u^r + iv$ does not contain any affine point of $\mathcal{B}.$
There is exactly one point of the line $\ell_{\infty}$ in $\mathcal{B}.$
Thus $\mathcal{B}$ consists of $\left( q+\frac{(q-1)q}{r}\right) q +1$ points.   

First consider the case $q$ odd.
Let $\ell$ be a non-horizontal affine line with equation $X = nY + d$. Let $n=n_1+in_2$ and 
$d=d_1+id_2$ where $n_1,n_2,d_1,d_2 \in \mathbb{F}_q$ Then
$$\ell \cap \mathcal{B} = \{ (x, y) \in  \mathcal{U} \mid \, y = u^r + iv, \, 2n_1u^r + 2hn_2v + 2c_1 + u^{2r} - hv^2 = 0, u, v  \in \mathbb{F}_q\}.$$
Thus affine points $(x,u^r+iv)$ of $\mathcal{B}$ on the line $\ell $ correspond to points of the 
curve having affine equation
$$\mathcal{A}_{\mathrm{odd}}: \quad 2n_1U^r + 2hn_2V + 2d_1 + U^{2r} - hV^2 = 0.$$
This curve is the affine part of the curve $\mathcal{X}_{\mathrm{odd}}$ in $\mathrm{PG}(2,q).$
Suppose that $\ell$ is not a tangent line to $\mathcal{U}$. Then
$n^{q+1} \neq d^q +d.$ It holds if and only if $n_1^2\neq 2d_1 + h n_2^2.$
So in this case 
by Lemma \ref{Lemma_odd}, $\mathcal{X}_{\mathrm{odd}}$ is absolutely irreducible and its genus is equal to $r-1$.

If $q$ is even and
$\ell$ is a non-horizontal affine line with equation $X = nY + d,$ $n=n_1+in_2$ and 
$d=d_1+id_2$ where $n_1,n_2,d_1,d_2 \in \mathbb{F}_q$ then
$$\ell \cap \mathcal{B} = \{ (x, y) \in  \mathcal{U} \mid \, y = u^r + iv,  \, 
n_2u^r + (n_1+n_2)v + d_2 + u^{2r} +u^rv + hv^2 = 0, u, v  \in \mathbb{F}_q\}.$$
Thus affine points $(x,u^r+iv)$ of $\mathcal{B}$ on the line $\ell $ correspond to points of the 
curve having affine equation
$$\mathcal{A}_{\mathrm{even}}: \quad n_2U^r + (n_1+ n_2)V + d_2 + U^{2r} +U^rV + hV^2 = 0.$$
This curve is the affine part of the curve $\mathcal{X}_{\mathrm{even}}$ in $\mathrm{PG}(2,q).$
Suppose that $\ell$ is not a tangent line to $\mathcal{U}$. Then
$n^{q+1} \neq d^q +d.$ It holds if and only if 
$n_1^2\neq n_1n_2 + h n_2^2 +d_2.$
So in this case 
by Lemma \ref{Lemma_even}, $\mathcal{X}_{\mathrm{even}}$ is absolutely irreducible and its genus is at most to $r-1$.

For all $q,$ the Hasse-Weil bound implies that each of the curves $\mathcal{X}_{\mathrm{odd}}$ 
and $\mathcal{X}_{\mathrm{even}}$ has at least
$q+1-2(r-1)\sqrt{q}$ points in $\mathrm{PG}(2,q).$
Both curves have a unique point at infinity, $(0:1:0).$
Now consider the points of the curves $\mathcal{A}_{\mathrm{odd}}$ and
$\mathcal{A}_{\mathrm{even}}.$
The line $U=0$ contains at most two points of these curves.
If $u\neq 0$, $\xi$ is an $r$-th root of the unity and $(u,v)$ is on $\mathcal{A}_{\mathrm{odd}}$ or 
on $\mathcal{A}_{\mathrm{even}}$ then
the point $(\xi u,v)$ is also on $\mathcal{A}_{\mathrm{odd}}$ or 
on $\mathcal{A}_{\mathrm{even}}$, respectively.
But if $\epsilon $ is a primitive $r$-th root of the unity, then for $i=0,1,\dots r-1$ 
the points $(u,v)$ and $(\epsilon ^i u,v)$ of the curves  
give the same affine point $(x,y)=(x,u^r+iv)$ of $\ell \cap \mathcal{B}$. 
Hence $\ell \cap \mathcal{B}$ contains at least 
$$\frac{(q +1 -(2r -2)\sqrt{q})-1}{r}$$ points, and 
by the assumption on $r$ this number is greater than $2.$ 
Thus the set $\ell \cap \mathcal{B}$ contains at least two points.

All horizontal lines pass through $(0:1:0)\in \mathcal{B}$. Since no horizontal line is a tangent line to $\mathcal{U}$, 
it is sufficient to add one point for each of the lines with equation $Y= y_0\neq y = u^r + iv$ 
to extend $\mathcal{B}$ to a 2-blocking set of $\mathcal{H}_q$.
Let $\mathcal{S}_0$ be the set obtained from $\mathcal{B}$ by adding these extra points. Then the size of $\mathcal{S}_0$ is 
$$\underbrace{\left(q+\frac{(q-1)q}{r}\right)q +1}_{|\mathcal{B}|}+\underbrace{\left[ q^2-\left(q+\frac{(q-1)q}{r}\right)\right]}_{\textrm{extra points on horizontal lines}}=\left(q+\frac{(q-1)q}{r}\right)(q-1) +q^2+1.$$
Now the theorem follows from Corollary \ref{2-blocking}.
\endproof

The lower bound on the size of the semioval in Theorem \ref{Th:Main} depends on  
the divisor $r$ of $q-1.$ The greater $r$ the smaller  the size of $\mathcal{S}_0$, but the method works only if $r<\sqrt{q}/2$ holds.
Note that if we take $r=1$ in Theorem \ref{Th:Main} then the semioval obtained is just the Hermitian curve itself. 
If $1<r$ then the condition $r<\sqrt{q}/2$ implies $q>16.$

Also, if $q$ is even or odd, then sometimes $q-1=p$ or $q-1=2p$, respectively, where $p$ is a prime number. Hence there is no $r\neq 1$ 
(e.g. in the cases $q=32,128$) or the best possible value is $r=2$ (e.g. in the cases $q=23,27$). 
In the case $r=2$ Theorem \ref{Th:Main} gives 
semiovals of sizes 
$$\frac{q^3+2q^2-q+2}{2}\leq k \leq q^3+1.$$
This is the same as the result of Theorem \ref{th:terzo}.

The situation is much better if $q$ is an odd square. In the case $q=s^2,$ $s$ odd,
one can always choose $r=(s-1)/2$: this is the greatest possible divisor of $q-1$ 
satisfying the condition $r<\sqrt{q}/2.$ In this case Theorem \ref{Th:Main} has the following 
\begin{corollary}
\label{best}
Let $q$ be an odd square. Then in
$\mathrm{PG}(2,q^2)$ there exists a semioval $\mathcal{S}\subset \mathcal{H}_q$ of size $k$ 
for all $k$ satisfying the inequalities 
$$2q^2\sqrt{q}+4q^2-2q\sqrt{q}-3q+1\leq k \leq q^3+1.$$
\end{corollary}

If $q\geq 49$, then this extend the spectrum of sizes of constructed semiovals
in $\mathrm{PG}(2,q^2)$ significantly, because previously the corresponding lower bound
was $(q^3+q^2)/2$  (see \cite{KMP2010}).

If $q=s^t,$ $s$ odd and $t>2$ then $r=s-1$ is always a possible choice. In this case
the size of the smallest semioval we get is roughly $q^2\sqrt[t]{q^{t-1}}$.

\section{The proof of existence of smaller semiovals}

If $q$ is odd then
we can prove the existence of much smaller semiovals using a theorem about
dominating sets of bipartite graphs. Let $A$ and $B$ be the two vertex subsets
of a bipartite graph. We say that a vertex $v\in B$ dominates the subset
$S\subset A,$ if for any $s\in S$ there is an edge between $v$ and $s.$ A
subset $B'\subset B$ is a dominating set, if for any $a\in A$ there exists
$b'\in B'$ which dominates $a.$ 
The following lemma is due to S. K. Stein, the proof can be found e.g. in
\cite{gsz}.

\begin{lemma}
\label{veletlen}
Let $A$ and $B$ be the two vertex subsets of a bipartite graph. Denote by $d$
the minimum degree in $A.$ If $A$ has at least two elements, then there is a
set $B'\subset B$ dominating the vertices of $A$ with
$$|B'|\leq \left\lceil |B|\frac{\log (|A|)}{d}\right\rceil, $$
where $\log $ denotes natural base logarithm.
\end{lemma}

\begin{theorem}\label{ProbConstr}
Let $q>27$ be odd and $k$ be an integer satisfying 
\[
(q^2-q+1)\left\lceil \frac{8(q+1)}{q-1}\log q\right\rceil \leq k\leq q^3+1,
\]
where $\log $ denotes the natural base logarithm.
Then $\mathrm{PG}(2,q^2)$ contains semiovals of size $k.$
\end{theorem}

\begin{proof}
The Hermitian curve $\mathcal{H}_q$ is the disjoint union of $q+1$ 
$(q^2-q+1)$-arcs. Let $\mathcal{C}=\{ \mathcal{C}_{1}, \mathcal{C}_{2},\dots ,\mathcal{C}_{q+1}\} $ 
denote the set of these arcs.
If $P$ is a point of $\mathcal{H}_q$ that belongs to $\mathcal{C}_{i}$, then
the set of the $q+1$ tangents to $\mathcal{C}_{i}$ at $P$ contains the unique
tangent to $\mathcal{H}_q$ at $P,$ while each of the remaining $q$ lines
is a tangent to exactly one another element of $\mathcal{C}$; see \cite[page 10]{BJKRSSV2015}.
Thus any bisecant of $\mathcal{C}_{i}$ is a tangent 
to either $0$ or $2$ other arcs from $\mathcal{C}$,
hence is a bisecant of $(q-1)/2$ or $(q+1)/2$ elements of $\mathcal{C}.$ 

We define a bipartite graph with two vertex subsets $A$ and $B.$
Let the vertices in $B$ be the $(q^2-q+1)$-arcs giving the decomposition of $\mathcal H$, 
and the vertices
in $A$ be those lines that are not tangents to $\mathcal H$.
Let $a\in A$ and $b\in B$ be joined if and only if the corresponding
line is a bisecant of the corresponding arc.
Then $|B|=q+1,$ $|A|=q^4-q^3+q^2$ and $d=(q-1)/2.$ Hence from
Lemma \ref{veletlen} we get that there exists $B'\subset B$ dominating $A,$
and
$$|B'|\leq \left\lceil |B|\frac{\log (|A|)}{d}\right\rceil =
\left\lceil (q+1)\frac{\log (q^4-q^3+q^2)}{(q-1)/2}
\right\rceil \leq \left\lceil \frac{8(q+1)}{q-1}\log q\right\rceil .$$
Thus there exists a subset of $\left\lceil \frac{8(q+1)}{q-1}\log q\right\rceil $ arcs
${\mathcal V}\subset {\mathcal U}$ such that
each secant of $\mathcal H$ meets $\mathcal V$ in at least two points.  
Hence $\mathcal V$ is a semioval of size 
\begin{equation}\label{Bound2}
k_0=(q^2-q+1)\left\lceil \frac{8(q+1)}{q-1}\log q\right\rceil .
\end{equation}

If $k_0<k\leq q^3+1$ then Corollary \ref{2-blocking} guarantees the existence of a semioval of size $k.$
\end{proof}

In \cite{KMP2010} it was proved that there exist semiovals of sizes greater than 
\begin{equation}\label{Bound1}
q^2\left\lceil \frac{8q}{q+1}\log q\right\rceil + 1.
\end{equation}
The bound \eqref{Bound2} is smaller than the bound \eqref{Bound1} for infinitely many $q$. In fact, this happens when 
$$\left\lceil \frac{8q}{q+1}\log q\right\rceil=\left\lceil \frac{8(q+1)}{q-1}\log q\right\rceil.$$
Let $e=\lfloor \log q\rfloor$ and $f=\log q -\lfloor \log q\rfloor$. The previous equality is satisfied whenever 
$$\frac{e}{q}\leq f < \frac{q-1}{8(q+1)}-\frac{2e}{q+1}.$$
If $q$ is large enough, then $\frac{e}{q}=\epsilon$, with $\epsilon$ close to zero, and $\frac{q-1}{8(q+1)}-\frac{2e}{q+1}$ is greater than $1/9$. Therefore, for all $q$ such that 
$$ \epsilon \leq \log q -\lfloor \log q\rfloor <\frac{1}{9}$$
our bound is better than the previously known (the smallest prime power satisfying this condition is $q=137$).

\section{$2$-blocking sets of the Hermitian curve}
In this section we present a construction of an infinite family of $2$-blocking sets of the Hermitian curve which is a modification of the construction of $1$-blocking sets of the Hermitian curve presented in \cite{BJKRSSV2015}. Here we use the description of the Hermitian curve as given in \eqref{eq:HC}.
\begin{proposition}
Let $q$ be an odd prime power and $r$ be a divisor of $q-1$ such that $1<r <\frac{\sqrt{q}}{2}-\frac{2\log_2 q +1}{4}$. Then there exists a $2$-blocking set of the Hermitian curve for some value of $k$ satisfying the inequalities 
$$\frac{q^3-3q^2-2q}{r}+2q^2-q+2-2\lceil \log_2 q+1\rceil \leq k \leq \frac{q^3-3q^2-2q}{r}+2q^2+q+2+2\lceil \log_2 q+1\rceil .$$

\end{proposition}
\proof
Let $\mathcal{B}$ defined by
$$
\mathcal{B} = \{(x, y) \in  \mathcal{U} \mid y = u^r + iv, u, v \in \mathbb{F}_q \} \cup  \{(1 : 0 : 0)\}\subset \mathcal{H}_q.
$$
As in Theorem \ref{Th:Main} it is possible to prove that each non-horizontal line, which is not a tangent line of the Hermitian curve, intersects $\mathcal{B}$ in at least $\frac{q-1-(2r-2)\sqrt{q}}{r}$ and in at most $\frac{q-1+(2r-2)\sqrt{q}}{r}$ points. The horizontal lines are either blocked only by $(1 : 0 : 0)$ or they are fully contained in $\mathcal{B}$. Consider a $(q^2-q+1)$-arc $\mathcal{C}$ through $(1 : 0 : 0)$ contained in $\mathcal{U}$. Among the $q^2$ horizontal lines through $(1 : 0 : 0)$, $q$ of them intersect $\mathcal{C}$ only in $(1 : 0 : 0)$, while the other $q^2-q$ contain an extra point of $\mathcal{C}$ other than $(1 : 0 : 0)$. 
Consider 
$$\overline{\mathcal{B}}:= (\mathcal{B} \Delta \mathcal{C}) \cup \{(1 : 0 : 0)\},$$
where $\Delta$ indicates the symmetric difference of the two sets. The set $\overline{\mathcal{B}}$ still blocks all the non-horizontal lines and they are not completely contained in $\overline{\mathcal{B}}$, since on each of these lines at most two points are deleted from $\mathcal{B}$ or added to $\mathcal{B}$. Also, at most $q$ horizontal lines, namely the $q$ unisecant lines to $\mathcal{C}$ through  $(1 : 0 : 0)$, either intersect $\overline{\mathcal{B}}$ only in $(1 : 0 : 0)$ or they are fully contained in $\overline{\mathcal{B}}$. 

Consider  $\mathcal{C}_1, \ldots, \mathcal{C}_{q+1}$ the $(q^2-q+1)$-arcs partitioning $\mathcal{U}$ and let $\ell_1,\ldots,\ell_k$ be  $k\leq q$ horizontal lines. There exists at least a $(q^2-q+1)$-arc $\mathcal{C}_i$ contained in $\mathcal{U}$ intersecting at least $\frac{k}{2}$ of such lines. On the contrary, suppose that all the $(q^2-q+1)$-arcs contained in $\mathcal{U}$ intersect at most $\frac{k}{2}-1$ of such lines. Then, each arc contains at most $k-2$ points of $\mathcal{U} \cap \left(\ell_1\cup \cdots\cup\ell_k\right)$. Since $\mathcal{U} \cap \left(\ell_1\cup \cdots\cup\ell_k\right)=kq+1$, then there should exist at least $\frac{kq+1}{k-2}>q+1$ of such arcs contained in $\mathcal{U}$ and pairwise disjoint. This is a contradiction. 

Arguing as before we can prove that, given $k$ horizontal lines, there exist at most $\lceil \log_2 k+1\rceil $  arcs among $\mathcal{C}_1, \ldots, \mathcal{C}_{q+1}$ such that their union intersects all the  $k$ lines. Note that in this case each of these lines contains at most $2\lceil \log_2 k+1\rceil$ points from the union of the $(q^2-q+1)$-arcs. 

Let $\ell_1,\ldots,\ell_{k_1}$ and $r_1,\ldots,r_{k_2}$ be the horizontal lines intersecting $\overline{\mathcal{B}}$ only in $(1:0:0)$ and in $q+1$ points, respectively. From above we know that $k_1+k_2\leq q$. Let $\mathcal{C}_{i_1}, \ldots, \mathcal{C}_{i_{j_1}}$, $j_1=\lceil \log_2 k_1+1 \rceil$, be the $(q^2-q+1)$-arcs intersecting $\ell_1,\ldots,\ell_{k_1}$ and let $\mathcal{C}_{h_1}, \ldots, \mathcal{C}_{h_{j_2}}$, $j_2=\lceil \log_2 k_2+1 \rceil$, be the $(q^2-q+1)$-arcs intersecting $r_1,\ldots,r_{k_2}$. In particular, let 
$$\mathcal{B}_1 := \left(\ell_1\cup \cdots\cup \ell_{k_1} \right) \cap \left( \mathcal{C}_{i_1}, \ldots, \mathcal{C}_{i_{j_1}}\right)$$
and 
$$\mathcal{B}_2 := \left(r_1\cup \cdots\cup r_{k_2}\right) \cap \left( \mathcal{C}_{h_1}, \ldots, \mathcal{C}_{h_{j_2}}\right).$$
We have that $|\mathcal{B}_1| \leq 2\lceil \log_2 k_1+1\rceil $ and $|\mathcal{B}_2| \leq 2\lceil \log_2 k_2+1\rceil $. The set 
$$\widetilde{\mathcal{B}} = \left(\overline{\mathcal{B}} \setminus \mathcal{B}_2\right) \cup  \mathcal{B}_1 \cup \{(1:0:0)\}$$
intersects each horizontal line in at least two and in at most $q$ points. Also, since a non-horizontal line $\ell$ intersects $\overline{\mathcal{B}}$ in $t$ points, where
$$2\lceil \log_2 q+1\rceil +2<\frac{q-1-(2r-2)\sqrt{q}}{r}\leq t\leq \frac{q-1+(2r-2)\sqrt{q}}{r}<q-2\lceil \log_2 q+1\rceil ,$$
then $\ell$ intersects $\widetilde{\mathcal{B}}$ in $\widetilde t$ points, where 
$$ 2<\frac{q-1-(2r-2)\sqrt{q}}{r}\leq \widetilde t\leq \frac{q-1+(2r-2)\sqrt{q}}{r}<q.$$
This proves that $\widetilde{\mathcal{B}}$ is a $2$-blocking set of the Hermitian curve not containing any block of it.

Finally, note that the size of $\mathcal{B}$ is $\left(q+\frac{(q-1)q}{r}\right)q +1$ and that the points of $\mathcal{B}$ lie on $s=q+\frac{q(q+1)}{r}$ horizontal lines. Therefore the size of $\overline{\mathcal{B}}$ satisfies
$$\left(q+\frac{(q-1)q}{r}\right)q +1 +(q^2-q-2s+1)\leq |\overline{\mathcal{B}}| \leq \left(q+\frac{(q-1)q}{r}\right)q +1 +(q^2+q-2s+1).$$
To obtain $\widetilde{\mathcal{B}}$ we add or delete at most $2\lceil \log_2 q+1\rceil $ points. So,
$$\left(q+\frac{(q-1)q}{r}\right)q +1 +(q^2-q-2s+1)-2\lceil \log_2 q+1\rceil \leq |\widetilde{\mathcal{B}}|\leq$$
$$\leq \left(q+\frac{(q-1)q}{r}\right)q +1 +(q^2+q-2s+1)+2\lceil \log_2 q+1\rceil .$$
\endproof

\begin{flushleft}
Daniele Bartoli\\
Department of Mathematics,\\
Ghent University,\\
Krijgslaan 281, 9000 Ghent, Belgium\\
e-mails: {\sf dbartoli@cage.ugent.be}
\end{flushleft}

\begin{flushleft}
Gy\"orgy Kiss \\
Department of Geometry and
MTA-ELTE GAC Research Group \\
E\"otv\"os Lor\'and University \\
1117 Budapest, P\'azm\'any s. 1/c, Hungary \\
e-mail: {\sf kissgy@cs.elte.hu}
\end{flushleft}

\begin{flushleft}
Stefano Marcugini and Fernanda Pambianco \\
Dipartimento di Matematica e Informatica,
Universit\`{a} degli Studi di Perugia \\
Via Vanvitelli 1, 06123 Perugia, Italy \\
e-mails: {\sf gino@dmi.unipg.it}, {\sf fernanda@dmi.unipg.it}
\end{flushleft}

\end{document}